\documentclass[11pt,twoside]{amsart}
\usepackage{latexsym,amssymb,amsmath}

\numberwithin{equation}{section}
\newtheorem{theorem}{Theorem}[section]
\newtheorem{lemma}[theorem]{Lemma}

\theoremstyle{definition}
\newtheorem{definition}[theorem]{Definition}

\newcommand{\NN}{\mathbb{N}}
\newcommand{\ZZ}{\mathbb{Z}}

\newcommand{\PP}{\mathbb{P}}
\newcommand{\TT}{\mathbb{T}}
\newcommand{\xx}{\mathbf{x}}

\newcommand{\G}{\mathcal{G}}
\renewcommand{\H}{\mathcal{H}}

\newcommand{\supp}{\operatorname{supp}}

\newcommand{\reg}{\operatorname{reg}}

\newcommand{\wt}{\operatorname{wt}}

\theoremstyle{theorem}

\newcommand{\set}[1]{\left\{ #1 \right \} }

\newcommand{\rt}{\rightarrow}

\def\noqed{\renewcommand{\qedsymbol}{}}

\begin{document}


\title[Ideals over complete multipartite graphs]{Vanishing ideals over\\ complete multipartite graphs}

\thanks{The first author was partially funded by CMUC, through
European program COMPETE/FEDER and FCT project PEst-C/MAT/UI0324/2011.
Part of this work was developed during a research visit to 
to CINVESTAV of the IPN, M\'exico, in which the first author 
benefited from the financial support of a research grant from Santander Totta Bank (Portugal). 
The second author is a member of the 
Center for Mathematical Analysis, Geometry, and Dynamical Systems, Departamento de Matematica, 
Instituto Superior Tecnico, 1049-001 Lisboa, Portugal.}

\author{Jorge Neves}
\address[Jorge Neves]{CMUC, Department of Mathematics, University of Coimbra.\hfill\newline\hfill \indent Apartado 3008 - EC Santa Cruz,
3001-501 Coimbra, Portugal.}
\email{neves@mat.uc.pt}

\author{Maria Vaz Pinto}
\address[Maria Vaz Pinto]{
Departamento de Matem\'atica\\
Instituto Superior T\'ecnico\hfill\newline
\indent Universidade T\'ecnica de Lisboa\\
Avenida Rovisco Pais, 1\\
1049-001 Lisboa, Portugal
}
\email{vazpinto@math.ist.utl.pt}


\begin{abstract} 
We study the vanishing ideal of the parametrized algebraic toric set associated to the 
complete multipartite graph $\G=\mathcal{K}_{\alpha_1,\dots,\alpha_r}$ over a finite field 
of order $q$. We give an explicit family of binomial generators for this lattice ideal, consisting
of the generators of the ideal of the torus, (referred to as type I generators),
a set of quadratic binomials corresponding to the cycles of length $4$ in $\G$ and 
which generate the \emph{toric algebra of $\G$} (type II generators)  and a set of binomials of degree $q-1$ 
obtained combinatorially from $\G$ (type III generators). Using this explicit family of generators of the 
ideal, we show that its Castelnuovo--Mumford regularity is equal to $\max\set{\alpha_1(q-2),\dots,\alpha_r(q-2),
\lceil (n-1)(q-2)/2\rceil}$, where $n=\alpha_1+\dots + \alpha_r$.
\end{abstract}

\maketitle


\section{introduction}\label{sec: introduction}

The class of vanishing ideals of parameterized algebraic toric sets over a finite field was first studied by
Renteria, Simis and Villarreal in \cite{algcodes}. Here we focus on the case when 
the set is parameterized by the edges of a simple graph. 
Let $K$ be a finite field of order $q$ and $\G$ a simple graph with $n$ vertices $\set{v_1,\dots,v_n}$ and nonempty edge
set. Given a choice of ordering of the edges, given by a bijection
$e\colon \set{1,\dots,s} \rt E(G)$, 
and writing $\xx^{e(i)}=x_jx_k$ for every $\xx=(x_1,\dots,x_n)\in (K^*)^n$ and $e(i)=\set{v_j,v_k}\in E(\G)$, 
we define the associated \emph{algebraic toric
set} as the subset of $\PP^{s-1}$ given by:
\begin{equation}\label{eq: the algebraic toric set}
X = \set{(\xx^{e_1},\dots,\xx^{e_s})\in \PP^{s-1} : \xx \in (K^*)^n},
\end{equation}
where we abbreviate the notation $e(i)$ to $e_i$.

The variety $X$ can also be seen as the subgroup of $\TT^{s-1}\subset \PP^{s-1}$ given by the 
image of the group homomorphism $(K^*)^n \rt \TT^{s-1}$ defined by $\xx \mapsto (\xx^{e_1},\dots,\xx^{e_s})$.
The \emph{vanishing ideal of $X$}, which we denote by  $I(X)$, is the ideal generated by all homogeneous forms in
\mbox{$S=K[t_1,\dots,t_s]$} that vanish on $X$. This ideal is a Cohen--Macaulay, radical, lattice ideal of codimension 
$s-1$ (\emph{cf.}~\cite[Theorem 2.1]{algcodes}). 
One motivation for the study of these ideals lies in the fact that they 
combine the toric ideal of the edge subring of a graph, $P(\G)\subset K[t_1,\dots,t_s]$, 
with the arithmetic of the finite field. This relation is expressed in the equality:
\begin{equation}\label{eq: I in terms of the toric ideal}
I(X) = \bigl( \bigl [ P(\G)+(t_2^{q-1}-t_1^{q-1},\dots,t_{s}^{q-1}- t_1^{q-1}) \bigr ] : (t_1\cdots t_s)^\infty \bigr) 
\end{equation}
which (in particular) holds for any connected \emph{or} bipartite $\G$ (\emph{cf.}~\cite[Corollary~2.11]{algcodes}). 
Recall that $P(\G)$ is the kernel of the epimorphism $K[t_1,\dots,t_s]\rt K[\G]$ given by $t_i\mapsto y^{e_i}$, where
$K[\G]$ is the edge subring of $\G$, \emph{i.e.}, the subring of the polynomial ring $K[y_1,\dots,y_n]$ given by 
\mbox{$K[G]=K[y^{e_i} : i=1,\dots,s]$}. For a survey on the subject of toric ideals of edge subrings of graphs we
refer the reader to \cite[Chapter~5]{graphsrings}. 
\smallskip

The properties of $I(X)$ (even in the general case of disconnected graphs) are reflected
in $\G$ and \emph{vice versa}. For one, the degree (or multiplicity) of $S/I(X)$ is equal to 
\begin{equation}\label{eq: degree of S/I}
\renewcommand{\arraystretch}{1.3}
\left \{ 
\begin{array}{l}
\bigr(\frac{1}{2}\bigl)^{\gamma-1}(q-1)^{n-m+\gamma-1}{}, \text{ if }
\gamma\geq 1\text{ and }q \text{ is odd}, \\
(q-1)^{n-m+\gamma-1}{}, \text{ if }
\gamma\geq 1\text{ and }q \text{ is even}, \\
(q-1)^{n-m-1}, \text{ if }\gamma=0,
\end{array}
\right.
\end{equation}
where $m$ is the number of connected components of $\G$, of which exactly $\gamma$ are non-bipartite 
(\emph{cf}.~\cite[Theorem 3.2]{evencycles}). 
Another invariant of interest is the \emph{index of regularity} of the quotient $S/I(X)$, which, 
since this quotient is Cohen--Macaulay of dimension $1$, coincides with the Castelnuovo--Mumford regularity. 
To present day knowledge, there is no single general formula expressing the regularity of $S/I(X)$ in terms 
of the data of $\G$. It is known that when $\G=\mathcal{C}_{2k}$, an even cycle of length $2k$, 
the regularity of $S/I(X)$ is $(k-1)(q-2)$ (\emph{cf.}~\cite[Theorem~6.2]{evencycles}). In the case of an odd cycle, $X$ coincides
with $\TT^{s-1}$ (\emph{cf.}~\cite[Corollary~3.8]{algcodes}) --- another way of seeing this, using 
(\ref{eq: I in terms of the toric ideal}), is that
if $\G$ is an odd cycle then $P(\G)=(0)$ (\emph{cf}.~\cite{reesdegree}); accordingly 
$I(X)=(t_2^{q-1}-t_1^{q-1},\dots,t_{s}^{q-1}- t_1^{q-1})$ is a  complete intersection 
(see also~\cite[Theorem~1]{GRH}). In this case the 
regularity is $(s-1)(q-2)=(n-1)(q-2)$, where $n$ (odd) is the number of vertices (and edges) of $\G$ 
(\emph{cf.}~\cite[Lemma~1]{GRH}).
If $\G=\mathcal{K}_{a,b}$ is a complete bipartite graph, the regularity of $S/I(X)$ is given by
\[
\textstyle \max\set {(a-1)(q-2),(b-1)(q-2)}
\]
(\emph{cf.}~\cite[Corollary 5.4]{GR}). Recently, a formula for the regularity of $S/I(X)$ in the case of a complete 
graph $\G=\mathcal{K}_n$ was given in \cite{GRS2}. In this case, if $n>3$,
\begin{equation}\label{eq: formula for reg in complete case}
\textstyle \reg S/I(X) =\lceil(n-1)(q-2) /2 \rceil. 
\end{equation}
Notice that the case $\G=\mathcal{K}_2$ is trivial and case $\G=\mathcal{K}_3=\mathcal{C}_3$ was already discussed. 
\smallskip

In this work we focus on the case of $\G = \mathcal{K}_{\alpha_1,\dots,\alpha_r}$, a complete multipartite graph
with \mbox{$\alpha_1+\dots+\alpha_r=n$} vertices. One of our main results, Theorem~\ref{thm: regularity}, 
states that in this case, if $r\geq 3$ and $n\geq 4$,
\begin{equation}\label{eq: our regularity in the introduction}
\textstyle \reg S/I(X) = \max \set{\alpha_1(q-2),\dots,\alpha_r(q-2), \lceil (n-1)(q-2)/2 \rceil}. 
\end{equation}
This formula generalizes (\ref{eq: formula for reg in complete case}); it contains the case of 
the complete graph by setting $\alpha_1=\cdots = \alpha_r = 1$. However, as far as the proof of Theorem~\ref{thm: 
regularity}
is concerned, we restrict to the case when $\mathcal{K}_{\alpha_1,\dots,\alpha_r}$ is not a complete graph. Moreover,
the methods used in this work are distinctly orthogonal to those used in \cite{GRS2}. Our main interest being the lattice ideal
$I(X)$, we rely on a precise description of a generating set of binomials to prove the statement on the regularity. 
In Theorem~\ref{thm: generators}, we show that a given set of binomials generates $I(X)$. These binomials are 
classified into $3$ classes: the binomials $t_i^{q-1}-t_j^{q-1}$, for every $i\not = j$, which, 
by (\ref{eq: I in terms of the toric ideal}), 
belong to $I(X)$ no matter which $\G$ we take, and are referred to as \emph{type I generators}; the binomials
$t_it_j-t_kt_l\in P(\G)$, for each $e_ie_ke_je_l$ cycle of length $4$ contained in $\G$, are referred to as 
\emph{type II generators}; finally the \emph{type III generators}, obtained from weighted subgraphs of $\G$, 
are described in full detail in the beginning of Section~\ref{sec: generators}.
Theorem~\ref{thm: generators} applies without restrictions on $\alpha_1,\dots,\alpha_r$. In particular, it yields a 
generating set for $I(X)$ in the case of a complete bipartite graph, which, despite the result on the regularity 
of $S/I(X)$ in \cite{GR}, was missing in the literature.
\smallskip

This problem area has been attracting increasing interest. The field of binomial ideals has been quite
explored and its general theory can be found in Eisenbud and Sturmfels article ~\cite{ES}.  
In our present setting, these binomial ideals have a remarkable application to coding theory.
Associating to $X$ an evaluation code, one can relate two of its basic parameters (the length and the dimension) 
to $I(X)$ by a straightforward application of the Hilbert function (\emph{cf.}~\cite{Hansen1,Hansen2}); 
moreover a set of generators of $I(X)$ can make way to computing the Hamming distance of the code 
(\emph{cf.}~\cite{ci-codes}). 
There has been substantial recent research exploring the relation to 
coding theory (\emph{cf}.~\cite{evencycles,algcodes,ci-codes,vanishing}) and also
focusing on the vanishing ideal of parameterized algebraic toric sets (\emph{cf.}~\cite{LV,degentorus}).

Let us describe the structure of this paper. In Section~\ref{sec: prelimiaries} we establish the definitions and notations used
in the article. In that section, Lemma~\ref{lemma: Vanishing lemma} provides a useful characterization of a binomial
in $I(X)$ by a condition on the associated weighted subgraph of $\G$. In Section~\ref{sec: generators}, we 
describe $3$ families of binomials and prove that they form a generating
set for $I(X)$ --- Theorem~\ref{thm: generators}. 
In Section~\ref{sec: regularity}, we prove Theorem~\ref{thm: regularity},
that states that under the assumption that $r\geq 3$ and $n=\alpha_1+\cdots+\alpha_r \geq 4$, the regularity of
$S/I(X)$ is given by the integer $d$ of formula (\ref{eq: our regularity in the introduction}). We show this by:
i) exhibiting a monomial in $K[t_1,\dots,t_s]$ of degree $d$
and showing that that monomial does not belong to $I(X)+(t_1)$, where $t_1\in K[t_1,\dots,t_s]$ is a variable; 
ii) and by showing that every monomial in $K[t_1,\dots,t_s]$ of degree $d+1$ is in $I(X)+(t_1)$.

\smallskip

For any additional information in the theory of monomial ideals and Hilbert functions, we refer to 
\cite{Sta1, monalg}, and for graph theory we refer to \cite{Boll}. 

\smallskip

\emph{Acknowledgement.} The authors thank Rafael Villarreal for many helpful discussions.


\section{Preliminaries}\label{sec: prelimiaries}

Let $K$ be a finite field of order $q$. 
Throughout, $\G=\mathcal{K}_{\alpha_1,\dots,\alpha_r}$ will denote the complete multipartite graph. More precisely, a graph  
with vertex set $V_\G=\{v_1, \ldots, v_n\}$ endowed with a partition $V_\G=P_1\sqcup \cdots \sqcup P_r$, satisfying
$\# P_i = \alpha_i$, for $i=1,\dots,r$, and $\alpha_1 + \cdots + \alpha_r=n$, such that 
$\set{v_k,v_l}$ is an edge of $\G$ if and only if $v_k\in P_i$ and $v_l\in P_j$ with $i\not =j$.
An important case to consider is that when all $\alpha_i=1$, \emph{i.e.}, when $\G$ is 
the complete graph on $r$ vertices.
Fix an ordering of the set of edges of the graph, $E(\G)$,
given by $e \colon \set{1,\dots,s} \rt E(\G)$, where $s=((\sum_{i=1}^r \alpha_i)^2-\sum_{i=1}^r \alpha_i^2)/2$. We
write $e_j$ for $e(j)$. Let $S=K[t_1,\dots,t_s]$ be a polynomial ring on $s=\# E(\G)$ variables. We fix a bijection between 
the set of variables of $S$ and $E(\G)$, given by the map $t_i\mapsto e_i$. To ease notation, 
we also use $t_{\set{v_k,v_l}}$ for the variable corresponding to the edge $\set{v_k,v_l}$. 
Let $X\subset \PP^{s-1}$ be the algebraic toric set associated to $\G$, 
as defined in (\ref{eq: the algebraic toric set}), and denote by $I(X)\subset S$ its vanishing ideal.
\medskip

The identification of a monomial (or a binomial) of $S$ with a weighted subgraph of $\G$ will play an important role 
in Sections~\ref{sec: generators} and \ref{sec: regularity}. If $a=(a_1,\dots,a_s)\in \NN^s$ we denote by $t^a$ the monomial 
of $S$ given by $t_1^{a_1}\cdots t_s^{a_s}$.

\begin{definition}\label{def: weighted subgraph} 
Given $g=t^a\in S$, where $a\in \NN^s$, the \emph{weighted subgraph associated to $t^a$}, denoted by $\H_g$, is 
the subgraph of $\G$ with the same vertex set, with the set of edges corresponding to the variables dividing 
$t^a$ and with weight function given by $a$, \emph{i.e.}, such that the weight of the edge 
$e_i$ is $a_i$ for each $i\in \supp a$. The \emph{weighted degree} of a vertex $v$ is the sum of the weights of
all edges incident to $v$ in $\H_{g}$.
\end{definition}

Let us denote the weight, $a_i$, of an edge $e_i$ in $\H_{g}$ by $\wt_{\H_{g}}(e_i)= a_i$. 
We denote the weighted degree of $v$ by $\wt_{\H_g}(v)$. This 
number is zero if no edge in $\H_g$ is incident to $v$.

\begin{definition}
Given a binomial, $f=t^a-t^b\in S$, with $\supp a \cap \supp b=\emptyset$, the \emph{weighted subgraph
associated to $f$}, which we denote by $\H_f$, is defined by $\H_{t^a}\cup \H_{t^b}$. 
\end{definition}

The edges of $\H_f$ may be colored with two colors using the partition 
$E(\H_{t^a})\sqcup E(\H_{t^b})$. We will use a solid line 
for edges corresponding to variables dividing $t^a$ and a dotted line for the edges 
corresponding to variables dividing $t^b$. We refer to the former as black edges and to the latter as dotted edges.
Although one can also define the notion of weighted degree of a vertex in this case, 
we will only use it for monomials. The usual notion of degree of a vertex, \emph{i.e.}, the number of edges incident
to it, disregarding weights and coloring of the edges, will be used.

\begin{lemma}\label{lemma: Vanishing lemma}
Let $\G$ be any graph and $X$ its associated algebraic toric set.
Let $f=t^a-t^b$ be a binomial (not necessarily homogeneous). Then, 
$f$ vanishes on $X$ if and only if for all $v\in V_\G$, 
$\wt_{\H_{t^a}}(v)\equiv \wt_{\H_{t^b}}(v) \pmod{q-1}$.
\end{lemma}

\begin{proof}
\noqed
Suppose that $f=t^a-t^b$ vanishes on $X$ and let $v=v_{l_0} \in V_\G$, for $l_0\in \set{1,\dots,n}$. 
Let \mbox{$\xx=(\dots,x_lx_{l'},\dots)\in X$} be given by $x_{l_0}=u$, where $u$ is a generator of $K^*$, and $x_l=1$, 
for any $l\not = l_0$. Then 
\[\textstyle f(\xx)=0 \implies u^{\wt_{\H_{t^a}}(v_{l_0})} - u^{\wt_{\H_{t^b}}(v_{l_0})}=0 \iff 
 \wt_{\H_{t^a}}(v_{l_0})\equiv \wt_{\H_{t^b}}(v_{l_0}) \pmod{q-1}.\]
Conversely, let $\xx=(\dots,x_lx_{l'},\dots)\in X$ be any point of $X$ and assume that 
$\wt_{\H_{t^a}}(v)\equiv \wt_{\H_{t^b}}(v) \pmod{q-1}$, for every $v\in V_\G$.
Then 
\[
f(\xx)=\prod_{l=1}^n x_{l}^{\wt_{\H_{t^a}}(v_l)} - \prod_{l=1}^n x_{l}^{\wt_{\H_{t^b}}(v_l)}=0.
\;\;\;\;\;
\quad \square
\]
\end{proof}

Let $I\subset S=K[t_1,\dots,t_s]$ be (any) homogeneous ideal and let 
$H_{S/I}(d)=\dim_K S_d/I_d$, for every $d\geq 0$, be the Hilbert function of 
the graded ring $S/I$. 
There exists a polynomial $h_{S/I}(t)$ in $\ZZ[t]$ of degree $k-1$, where $k=\dim S/I$, 
such that $H_{S/I}(d)=h_{S/I}(d)$ for $d\gg 0$. The \emph{degree} or \emph{multiplicity} of $S/I$ is, by definition,
the leading coefficient of $h_{S/I}(t)$ multiplied by $(k-1)!$. The \emph{index of regularity} of $S/I$
is the least integer $\ell\geq 0$ such that $H_{S/I}(d)=h_{S/I}(d)$ for $d\geq \ell$. Letting
$b_{ij}=\mbox{dim}_K\operatorname{Tor}_i(K,S/I)_j$ be the \emph{graded Betti numbers} of $S/I$, the \emph{Castelnuovo--Mumford
regularity} of $S/I$ is, by definition, $\max\{j-i\vert\,b_{ij}\neq 0\}$. 
If $I$ is Cohen--Macaulay and $\dim S/I = 1$, then the index of regularity of 
$S/I$ coincides with the Castelnuovo--Mumford regularity of $S/I$ 
(cf.~\cite[Corollary~2.5.14 and Proposition~4.2.3 ]{monalg}).
As this is the context of this work, with $I=I(X)$ (\emph{cf.}~\cite[Theorem~2.2]{algcodes}), 
we will denote both by $\reg(S/I(X))$.
\smallskip

Hence, in our case, the Hilbert function of the ring $S/I(X)$ is constant for $d\geq \reg(S/I(X))$.
The constant value at which it stabilizes is $|X|$, and since the Hilbert function
of $S/I(X)$ is strictly increasing in the range $0\leq d \leq \reg(S/I(X))$ 
(cf.~\cite{duursma-renteria-tapia}, ~\cite{geramita-cayley-bacharach}),
the regularity of $S/I(X)$ is equal to the first $d$ for which it attains the value $|X|$. Note that
$|X|$ is equal to $(q-1)^{n-1}$ by the formula (\ref{eq: degree of S/I}) --- recall that $n=\# V_\G$.


\section{A generating set for $I(X)$}\label{sec: generators}

In this section we describe a generating set for the vanishing ideal $I(X)$ of the algebraic toric set
 associated to a complete multipartite graph $\G=K_{\alpha_1,\dots,\alpha_r}$. 
 We will distinguish $3$ types of generators. The \emph{type I generators}
are of the form $t_i^{q-1}-t_j^{q-1}$, for $1\leq i,j\leq s$.
The \emph{type II generators} are in $1$-to-$1$ correspondence with the cycles of length $4$, $\set{e_i,e_j,e_k,e_l}$, 
in $\G$ (\emph{cf.} see Figure~\ref{fig: 4-cycle in G}). Each cycle yields the generator: $t_it_k-t_jt_l$.

\begin{figure}[ht]
\begin{minipage}[b]{0.45\linewidth}
\centering
\begin{picture}(200,75)(-30,40)
\put(30,90){$\bullet$}
\put(19,90){$v_{i_1}$}
\put(62,112){$\bullet$}
\put(62,122){$v_{i_4}$}
\put(92,81){$\bullet$}
\put(100,81){$v_{i_3}$}
\put(70,49){$\bullet$}
\put(69,41){$v_{i_2}$}
\put(32.5,92){\line(1,-1){40}}
\put(43,65){$e_i$}
\multiput(72,49)(2,3){11}{$\cdot$}
\put(82,102){$e_k$}
\put(65,114){\line(1,-1){30}}
\put(88,61){$e_j$}
\multiput(32.5,92)(3,2){11}{$\cdot$}
\put(37,110){$e_l$}
\end{picture}
\caption{A $4$-cycle yielding type \emph{II} generator.}
\label{fig: 4-cycle in G}
\end{minipage}
\hspace{-1.5cm}
\begin{minipage}[b]{0.45\linewidth}
\centering
\begin{picture}(200,140)(-50,40)
\put(-20,90){$\bullet$}
\put(-37,87){$v_{l_{n'+1}}$}
\multiput(-19,90)(2,3){21}{$\cdot$}
\put(-10,130){$e_{j_1}$}
\multiput(-19,90)(3,1.8){20}{$\cdot$}
\put(6,118){$e_{j_2}$}
\multiput(-19,90)(4,.4){20}{$\cdot$}
\put(20,102){$e_{j_3}$}
\multiput(-19,90)(4,-1.5){27}{$\cdot$}
\put(18,67){$e_{j_{n'}}$}
\multiput(67,83)(2,-3){3}{$\cdot$}
\put(20,150){$\bullet$}
\put(13,160){$v_{l_1}$}
\put(38,125){$\bullet$}
\put(28,134){$v_{l_2}$}
\put(56,98){$\bullet$}
\put(47,108){$v_{l_3}$}
\put(87,50){$\bullet$}
\put(90,40){$v_{l_{n'}}$}
\put(23,152){\line(6,-1){80}}
\put(50,153){$e_{i_1}$}
\put(40,127){\line(6,1){64}}
\put(55,136){$e_{i_2}$}
\put(58,100){\line(6,5){47}}
\put(65,120){$e_{i_3}$}
\put(89.75,51.5){\line(1,6){14.5}}
\put(103,100){$e_{i_{n'}}$}
\put(102,136){$\bullet$}
\put(112,138){$v_{l_0}$}
\end{picture}
\caption{A subgraph yielding a type \emph{III} generator.}
\label{fig: third type of generator illustration}
\label{fig:figure2}
\end{minipage}
\end{figure}

A \emph{type III generator} is specified by the choice of $n'\geq 2$ distinct vertices, $v_{l_1},\dots, v_{l_{n'}}\in V_\G$,  
plus $2$ additional vertices $v_{l_0}$ and $v_{l_{n'+1}}$ such that the edges $\set{v_{l_0},v_{l_k}}$ and
$\{v_{l_{n'+1}},v_{l_k}\}$ exist, for all $1\leq k \leq n'$ and, furthermore, by the choice of positive integers
$1\leq d_k\leq q-2$ such that $d_1+\cdots + d_{n'} = q-1$. We associate to this data the binomial 
$\prod_{k=1}^{n'} t_{i_k}^{d_k}-\prod_{k=1}^{n'} t_{j_k}^{d_k}$, 
where, for each $k\in \set{1,\dots,n'}$,
$i_k$ and $j_k$ are such that $e_{i_k}=\set{v_{l_0},v_{l_k}}$ and $e_{j_k}=\{v_{l_{n'+1}},v_{l_k}\}$.
The associated weighted subgraph of $\G$ is depicted in Figure~\ref{fig: third type of generator illustration}.
By a straightforward application of Lemma~\ref{lemma: Vanishing lemma}, it is clear that 
the homogeneous binomials of the $3$ types listed earlier
belong to $I(X)$. 
\smallskip 

We will need the following lemma.

\begin{lemma}\label{lemma: generalised type III}
Let $v_{l_0}, v_{l_1}, \dots, v_{l_{n'}}, v_{l_{n'+1}}\in V_\G$,  
be such that $n' \geq 2$ and the edges $\set{v_{l_0},v_{l_k}}$ and
$\{v_{l_{n'+1}},v_{l_k}\}$ exist in $\G$, for all $1\leq k \leq n'$. Let 
$1\leq d_k\leq q-2$ be such that $d_1+\cdots + d_{n'} = \alpha(q-1)$, where $\alpha$ is a positive integer. 
Consider $f=\prod_{k=1}^{n'} t_{i_k}^{d_k}-\prod_{k=1}^{n'} t_{j_k}^{d_k}$, 
where, for each $k\in \set{1,\dots,n'}$,
$i_k$ and $j_k$ are such that $e_{i_k}=\set{v_{l_0},v_{l_k}}$ and $e_{j_k}=\{v_{l_{n'+1}},v_{l_k}\}$. Then $f$
belongs to the ideal generated by the type III binomials.
\end{lemma}

\begin{proof}
We proceed by induction on $\alpha$.
If $\alpha=1$, $f$ is exactly a type III generator. Assume $\alpha \geq 2$.
Choose $m\in \NN$ such that  
$2\leq m\leq n'$ and, for each $k=1,\dots,m$, choose $d'_k\in \NN$ 
such that $1\leq d'_k\leq d_k \leq q-2$ and $d'_1+\dots+d'_{m}=q-1$. 
Then $\prod_{k=1}^{m} t_{i_k}^{d'_k}-\prod_{k=1}^{m} t_{j_k}^{d'_k}$ is a type III generator. 
Write $f=t^a\prod_{k=1}^{m} t_{i_k}^{d'_k}-t^b\prod_{k=1}^{m} t_{j_k}^{d'_k}$, for appropriate
$a,b\in\NN^s$. Then:
\[
f= \textstyle \prod_{k=1}^{n'} t_{i_k}^{d_k}-\prod_{k=1}^{n'} t_{j_k}^{d_k} = 
 \left ( \prod_{k=1}^{m} t_{i_k}^{d'_k}-\prod_{k=1}^{m} t_{j_k}^{d'_k}\right )t^a + (t^a-t^b)\prod_{k=1}^{m} t_{j_k}^{d'_k}.
\]
By induction, $t^a-t^b$ is in the ideal generated by the type III binomials, hence so is $f$.
\end{proof}

\begin{lemma}\label{lemma: useful lemma}
Let $f=t^a-t^b$ be a (homogeneous) binomial in $I(X)$. If $\H_f$ contains one of the subgraphs depicted in 
Figure~\ref{fig: two dichromatic vertex arrangements}, with $v_1\not = v_2$, and either there is an edge in $\G$ 
through $v_1$ and $v_2$ 
or one of $v_1$, $v_2$ is dichromatic,
then there exists $j\in \set{1,\dots,s}$ and a (homogeneous) binomial $g\in K[t_1,\dots,t_s]$ such 
that $f-t_jg$ belongs to the ideal of $K[t_1,\dots,t_s]$ generated by the binomials 
of type \mbox{I,II and III.}
\begin{small}
\begin{figure}[h]
\begin{picture}(300,50)(1,75)
\put(19,110){$\bullet$}
\put(10,107){$v_1$}
\put(22,89){$e_i$}
\put(22,112){\line(2,-3){20}}
\put(40,80){$\bullet$}
\put(38,72){$v_{{4}}$}
\multiput(43,80)(3,0){18}{$\cdot$}
\put(65,72){$e_j$}
\put(92,80){$\bullet$}
\put(90,72){$v_{{3}}$}
\put(113,110){$\bullet$}
\put(115,112){\line(-2,-3){20}}
\put(119,107){$v_2$}
\put(108,89){$e_k$}
\put(179,110){$\bullet$}
\put(170,107){$v_1$}
\put(182,89){$e_j$}
\put(198,72){$v_{{4}}$}
\multiput(181,110)(2,-3){10}{$\cdot$}
\put(200,80){$\bullet$}
\put(203,83){\line(1,0){50}}
\put(225,72){$e_i$}
\put(252,80){$\bullet$}
\put(250,72){$v_{{3}}$}
\put(273,110){$\bullet$}
\multiput(274,110)(-2,-3){10}{$\cdot$}
\put(279,107){$v_2$}
\put(268,89){$e_l$}
\end{picture}
\caption{Two special dichromatic edge arrangements}
\label{fig: two dichromatic vertex arrangements}
\end{figure}
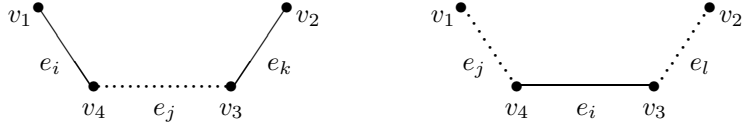
\end{small}
\end{lemma}

\begin{proof}
Let $J$ be the ideal of $K[t_1,\dots,t_s]$ generated by the binomials 
of type I, II and III.

Case 1. 
Suppose, without loss of generality, that it is the case of the graph on the left of Figure~\ref{fig: two 
dichromatic vertex arrangements} and that there exists an edge in $\G$ through $v_1$ and $v_2$. 
Denote this edge by $e_l$. Then $t_it_k-t_jt_l$ is a type II generator.
Let $f=t^a-t^b = t_it_kt^{a'}-t_jt^{b'}$, for appropriate $a',b'\in \NN^s$ and consider $g=t_lt^{a'}-t^{b'}$. Then,
\[
f-t_jg = t_it_kt^{a'}-t_jt^{b'} - t_j(t_lt^{a'}-t^{b'}) = (t_it_k-t_jt_l)t^{a'} \in J. 
\]
Case 2. Assume now that $v_1$ and $v_2$ have no edge in $\G$ between them, (in other words, that they belong 
to the same part of the partition of $V_\G$) and that $v_1$ is dichromatic. Let $e_l$ be a dotted edge incident 
to this vertex. Let $v_5$ be its other endpoint. If there exists an edge in $\G$ between $v_3$ and $v_5$, 
we reduce to Case 1 (graph on the right of Figure~\ref{fig: two dichromatic vertex arrangements}). 
Assume $v_3$ and $v_5$ belong to the same part of the partition of $V_\G$. 
Observe that $v_1$ and $v_3$ must be in distinct parts of the partition of $V_\G$ 
(otherwise $v_2$ and $v_3$ would be in that same part, but we are assuming that there is an edge $e_k$ through 
$v_2$ and $v_3$). Therefore there is an edge in $\G$, denote it by $e_\lambda$, through $v_1$ and $v_3$. 
By the same type of reasoning, we deduce that there is an edge in $\G$, denote it by $e_{\mu}$, through $v_2$ and $v_4$.
If $v_3=v_5$, then $e_\lambda=e_l$, which is a dotted edge in $\H_f$.
Since $v_2$ and $v_4$ are in different parts of the partition of $V_\G$, we are reduced (with $e_i, e_l$ and $e_k$)
to Case 1 (graph on the left of Figure~\ref{fig: two dichromatic vertex arrangements}).
So we may assume $v_3 \neq v_5$.
Write $f=t_it_kt^{a'}-t^b$ and consider $f'=t_{\lambda}t_{\mu}t^{a'}-t^{b}$. 
Notice that $e_\lambda$ is a black edge of the subgraph $\H_{f'}$, and 
since $t_l$ and $t_j$ divide $t^b$, $e_l$ and $e_j$ are dotted edges of the subgraph $\H_{f'}$. 
Observe also that $v_4$ and $v_5$ are not in the same part of the partition of $V_\G$ 
(otherwise $v_3$ and $v_4$ would be in that same part, but we are assuming that there is an edge $e_j$ through 
$v_3$ and $v_4$). Then there
must be an edge in $\G$ through $v_4$ and $v_5$.
Therefore, $\H_{f'}$ contains a subgraph with edges $e_l, e_\lambda$ and $e_j$, satisfying the assumption of Case 1
(graph on the right of Figure~\ref{fig: two dichromatic vertex arrangements}). 
Thus, there exists 
$g \in K[t_1,\dots,t_s]$ such that $f'-t_\lambda g$ is a multiple of a type II generator. To complete the proof 
it suffices to observe that $f = f' + (t_it_k-t_\lambda t_\mu)t^{a'}$ 
and that $t_it_k-t_\lambda t_\mu$ is a type II generator.
\end{proof}

\begin{theorem}\label{thm: generators}
The binomials of type I, II and III generate $I(X)$. 
\end{theorem}
\begin{proof}
Denote by $J$ the ideal of $K[t_1,\dots,t_s]$ generated by the binomials of type I, II and III.
Clearly $J \subseteq I(X)$.
By \cite[Theorem 4.5]{evencycles} there exists a set of generators of $I(X)$ consisting of the generators of type I, 
plus a finite set of homogeneous binomials $t^a-t^b$ with $\supp(a)\cap \supp(b)=\emptyset$ and such that
the degree of $t^a-t^b$ in each variable is $\leq q-2$. Hence it will suffice to show that any homogeneous 
binomial $f=t^a-t^b\in I(X)$ satisfying the latter conditions belongs to the ideal generated by the elements 
in the $3$ classes described, or, equivalently, is congruent to $0$ modulo $J$.
Let $f=t^a-t^b$ be such a binomial and, as defined above, let $\H_f$ be the induced subgraph of $\G$. 
We will argue by induction on $\# V^+_{\H_f} + \deg(f)$, where $V^+_{\H_f}$ denotes the subset of $V_{\H_f}=V_\G$
of vertices with positive degree. 
\medskip

By Lemma~\ref{lemma: Vanishing lemma}, no vertex of $\H_f$ has (standard) degree equal to $1$. Hence $\# V^+_{\H_f}\geq 3$ and if
$\# V^+_{\H_f}=3$ then $\H_f$ reduces to a triangle. This situation is impossible, since $f$ is homogeneous and 
the condition given in the statement of Lemma~\ref{lemma: Vanishing lemma}
must be satisfied. The base case is thus 
$\# V^+_{\H_f}=4$ and $\deg(f)=2$. $\H_f$ must reduce to a square, and using Lemma~\ref{lemma: Vanishing lemma} 
we deduce that $f$ is a type II generator if $q>3$ and a type II or a type III generator if $q=3$.
\medskip

Suppose that $\# V^+_{\H_f} + \deg(f) \geq 7$. 
If all vertices of $\H_f$ are dichromatic then we can find a subgraph of $\H_f$ satisfying the assumptions of
Lemma~\ref{lemma: useful lemma}. Then there exists $j\in \set{1,\dots,s}$ such that $f-t_jg \in J$, for some homogeneous 
binomial $g\in K[t_1,\dots,t_s]$. 
Since $J\subset I(X)$, $f\in I(X)$ and $t_j$ does not vanish on any point of $X$, 
we deduce that $g\in I(X)$. It is clear that $\deg(g)<\deg(f)$. We may assume that 
$g=t^{a'}-t^{b'}$ with $\supp(a')\cap \supp(b') = \emptyset$ (dividing through by an appropriate monomial,
if necessary). 
We may also assume that $0\leq a'_i,b'_i\leq q-2$\;: suppose that $t_l^{a'_l}$ divides $t^{a'}$ and
$a'_l \geq q-1$; let $t_k$ divide $t^{b'}$; then
$g=t_l^{q-1}t^{a''}-t_kt^{b''}+t_k^{q-1}t^{a''}-t_k^{q-1}t^{a''}=
(t_l^{q-1}-t_k^{q-1})t^{a''}+t_kh$, where $h=t_l^{a'_l-(q-1)}t^c-t^d$, for appropriate $a'',b'',c,d\in \NN^s$;
since $t_l^{q-1}-t_k^{q-1}$ is a type I generator, $g \in I(X)$ if and only if $h \in I(X)$. 
Repeating the argument, we may indeed assume that $0\leq a'_i,b'_i\leq q-2$.
Therefore, by the induction assumption,
$g\equiv 0 \mod{J}$, which implies that $f\equiv 0 \mod{J}$.
\smallskip

Let now $v_{l_0}$ be a monochromatic vertex of $\H_f$. Assume, without loss of generality, that all of its
incident edges are black. Denote by $v_{l_1},\dots,v_{l_{n'}} \;(n' \geq2)$ their endpoints and let 
$e_{i_k}=\set{v_{l_0},v_{l_k}}$, for $k=1,\dots,n'$.

Special condition: Suppose there exists
$v_{l_{n'+1}}\in V^+_{\H_f}$ with $v_{l_{n'+1}}\not = v_{l_0}$ and such that, for all $k=1,\dots,n'$, 
\mbox{$\{v_{l_{n'+1}},v_{l_k}\}\in E_\G$}.
Denote these edges by $e_{j_k}$.  By Lemma~\ref{lemma: Vanishing lemma}, $a_{i_1}+\dots +a_{i_{n'}}\equiv 0 \pmod{q-1}$ 
and therefore, by Lemma~\ref{lemma: generalised type III},
$\prod_{k=1}^{n'} t_{i_k}^{a_{i_k}} - \prod_{k=1}^{n'} t_{j_k}^{a_{i_k}} \in J$. Writing
$t^a = t^{a'}\prod_{k=1}^{n'} t_{i_k}^{a_{i_k}}$, for suitable $a'\in \NN^s$,
\[
\textstyle f= t^{a'}\prod_{k=1}^{n'} t_{i_k}^{a_{i_k}} -t^b \equiv t^{a'}\prod_{k=1}^{n'} t_{j_k}^{a_{i_k}} -t^b \mod J.
\]
Write $g=f-t^{a'}(\prod_{k=1}^{n'} t_{i_k}^{a_{i_k}} - \prod_{k=1}^{n'} t_{j_k}^{a_{i_k}})=
t^{a'}\prod_{k=1}^{n'} t_{j_k}^{a_{i_k}} -t^b=t^c-t^b$, for appropriate $c\in \NN^s$. 
Then $g \in I(X)$, the graph induced by $g$ has one vertex of positive weighted degree 
fewer than the graph $\H_f$, and $g$ has degree equal to $\deg f$.
As above, by using a type I generator and dividing through by 
an appropriate monomial (in which case the resulting monomial would have degree strictly less that $\deg f$), we 
may also assume that no variable in $g$ occurs to a higher power than $q-2$ and that 
$\supp(c)\cap \supp(b)=\emptyset$.
 By induction, we deduce that 
\mbox{$f\equiv g \equiv 0\mod J$.}
\smallskip

Denote now by $P_{\mu_0},P_{\mu_1},\dots,P_{\mu_m}$  
the parts of the partition of $V_\G$ that have nonempty intersection with $V^+_{\H_f}$, with $v_{l_0}\in P_{\mu_0}$. 
If $\# (P_{\mu_0}\cap V^+_{\H_f}) \geq 2$, we are in the special condition case, and therefore $f \equiv 0\mod J$.
We may assume that $P_{\mu_0}\cap V^+_{\H_f}=\set{v_{l_0}}$. 
We may also assume that $P_{\mu_k}\cap \set{v_{l_1},\dots, v_{l_{n'}}}\not = \emptyset$, 
for all \mbox{$k=1,\dots,m$}, for otherwise, if for some $k \in \{1, \ldots, m\}$, this intersection is empty, 
choosing $v_{l_{n'+1}} \in P_{\mu_k}\cap V_{\H^+_f}$,
we are again in the special condition case.
Another observation is the following: if there exists $k \in \{1, \ldots, m\}$ such that 
$\# (P_{\mu_k}\cap V_{\H^+_f}) \geq 2$ and one of the vertices of $P_{\mu_k}\cap V_{\H^+_f}$ is monochromatic, 
then, using the same argument we used for $P_{\mu_0}\cap V_{\H^+_f}$, we conclude that $f \equiv 0\mod J$. 
We may therefore assume that for all $k \in \{1, \ldots, m\}$ such that 
$\# (P_{\mu_k}\cap V_{\H^+_f}) \geq 2$, the vertices of $P_{\mu_k}\cap V_{\H^+_f}$ are all dichromatic.

Applying Lemma~\ref{lemma: Vanishing lemma} for each vertex in $V^+_{\H_f}\setminus \set{v_{l_0}}$ and 
summing all the congruences obtained,
\begin{equation}\label{eq: sum of congruences}
\textstyle \sum_{k=1}^{n'} a_{i_k} + 2\sum_{i^*} a_{i^*} = 2\sum_{i^*} b_{i^*} +\delta (q-1),
\end{equation}
where $i^*$ varies on the subset of $\set{1,\dots,s}$ corresponding to edges of $\H_f$ except for 
$e_{i_1},\dots, e_{i_{n'}}$, and $\delta \in \ZZ$ 
is given by $\delta= \sum_{v \in V^+_{\H_f}\setminus \set{v_{l_0}}} \delta_v$,
where $\delta_v \in \ZZ$ yields the congruence in of Lemma~\ref{lemma: Vanishing lemma} for the vertex $v$.
Since $f$ is homogeneous,
\begin{equation}\label{eq: homogeneity}
\textstyle \sum_{k=1}^{n'} a_{i_k} + \sum_{i^*} a_{i^*} = \sum_{i^*} b_{i^*}.
\end{equation}
Together, (\ref{eq: sum of congruences}) and (\ref{eq: homogeneity}) imply that 
$\textstyle \sum_{k=1}^{n'} a_{i_k} + \delta(q-1) = 0$.
We deduce that $\delta <0$. This means that for some vertex 
$v_{\lambda_0} \in V^+_{\H_f} \cap \bigl(P_{\mu_1} \cup \cdots \cup P_{\mu_m} \bigr)$, 
$\delta_{\lambda_0}<0$, and, in particular,
the sum of  the weights of dotted edges incident to $v_{\lambda_0}$ is $\geq q-1$. 
Denote by $e_{\nu_1},\dots,e_{\nu_{\tilde{n}}}$ the dotted edges incident to $v_{\lambda_0}$ and by 
$v_{\lambda_1},\dots,v_{\lambda_{\tilde{n}}}$ their endpoints, so that 
$e_{\nu_k}=\set{v_{\lambda_0},v_{\lambda_k}}$ and $\sum_{k=1}^{\tilde{n}} b_{\nu_k}\geq q-1$. 
Notice that, necessarily $\tilde{n}\geq 2$. We claim that $v_{\lambda_0}$ must be dichromatic. This is clear
if $v_{\lambda_0}$ coincides with one of $v_{l_1},\dots,v_{l_{n'}}$. Otherwise, since $v_{\lambda_0}$ 
belongs to some $P_{\mu_k}$ and $P_{\mu_k} \cap \set{v_{l_1},\dots,v_{l_{n'}}} \not = \emptyset$, 
we get $\#(P_{\mu_k}\cap V^+_{\H_f})\geq 2$ and hence $v_{\lambda_0}$ is not monochromatic.
The same argument can be used to show that $v_{\lambda_1},\dots,v_{\lambda_{\tilde{n}}}$ are all dichromatic. 
For each of the vertices $v_{\lambda_0},v_{\lambda_1},\dots,v_{\lambda_{\tilde{n}}}$ choose a black edge incident 
to it, denote it by $e_{\gamma_i}$, for $i=0,\dots,\tilde{n}$, and let $v_{\beta_i}$ be its endpoint.

Suppose there exists $r \in \{1, \ldots, \tilde{n}\}$ such that 
$v_{\beta_0} \not = v_{\beta_r}$. Either there is an edge in $\G$ through $v_{\beta_0}$ and
$v_{\beta_r}$, or $v_{\beta_0}$ and $v_{\beta_r}$ are in the same part of the partition of $V_\G$. Then there exists
$k \in \{1, \ldots, m\}$ such that $v_{\beta_0}, v_{\beta_r} \in P_{\mu_k}\cap V^+_{\H_f}$, and 
since $\#(P_{\mu_k}\cap V^+_{\H_f})\geq 2$, $v_{\beta_0}$ and $v_{\beta_r}$ are dichromatic. 
In any case, we may use Lemma~\ref{lemma: useful lemma} and argue as previously to show that $f\equiv 0 \mod J$. 
Suppose then that for all $r \in \{0, \ldots, \tilde{n}\}$, the endpoints 
$v_{\beta_r}$ all coincide.
As in the proof of Lemma~\ref{lemma: generalised type III}, 
choose $m'\in \NN$ such that  
$2\leq m'\leq \tilde{n}$ and, for each $k=1,\dots,m'$, choose $b'_{\nu_k}\in \NN$ 
such that $1\leq b'_{\nu_k}\leq b_{\nu_k} \leq q-2$ and $b'_{\nu_{1}}+\dots+b'_{\nu_{m'}}=q-1$. 
Then 
\[
\textstyle \prod_{k=1}^{m'} t_{\gamma_k}^{b'_{\nu_k}}-\prod_{k=1}^{m'} t_{\nu_k}^{b'_{\nu_k}} 
\]
is a type III generator, and therefore belongs to $J$.
Writing 
\[
\textstyle f=t^a-t^b = t^{a'}\prod_{k=1}^{m'}t_{\gamma_k}-t^{b'}\prod_{k=1}^{m'}t_{\nu_k}^{b'_{\nu_k}}, 
\]
for appropriate $a',b'\in \NN^s$, we deduce that 
\[
\textstyle f \equiv t^{a'}\prod_{k=1}^{m'}t_{\gamma_k}-t^{b'}\prod_{k=1}^{m'}t_{\gamma_k}^{b'_{\nu_k}}=
\bigl (t^{a'}-t^{b'}\prod_{k=1}^{m'} t_{\gamma_k}^{b'_{\nu_k}-1}\bigr )\prod_{k=1}^{m'}t_{\gamma_k} \mod J.
\]
The homogeneous binomial $g=t^{a'}-t^{b'}\prod_{k=1}^{m'} t_{\gamma_k}^{b'_{\nu_k}-1}=t^{a'}-t^{b''}$, 
which by the above belongs to $I(X)$, is such that $\deg(g)=\deg(f)-m' < \deg(f)$.
We may assume $\supp(a')\cap \supp(b'')=\emptyset$.
By induction, $g\equiv 0 \mod J$, which implies that $f\equiv 0 \mod J$ and completes the proof.
\end{proof}


\section{Regularity of $S/I(X)$}\label{sec: regularity}

Let $\G=\mathcal{K}_{\alpha_1,\dots,\alpha_r}$ be a complete multipartite graph with $r\geq 3$. We assume 
$\G$ does not coincide with $\mathcal{K}_{1,1,1}$, the complete graph on $3$ vertices. In that case, the 
associated toric set $X$ coincides with the ambient torus, $\TT^2\subset \PP^2$ and by \emph{cf.}~\cite[Lemma~1]{GRH},  
the regularity of $S/I(X)$ is $2(q-2)$.
In this section we will show that in all other cases,  
\begin{equation*}
\reg S/I(X)=\max \set{\alpha_1(q-2),\alpha_2(q-2),\dots,\alpha_r(q-2),\left \lceil (n-1)(q-2)/2 \right \rceil},
\end{equation*}
where, $n=\# V_\G = \alpha_1+\cdots +\alpha_r\geq 4$. If, without loss in generality, we assume that $\alpha_1\geq \alpha_2\geq \cdots \geq \alpha_r$,
then this formula takes on the simpler form:
\[
\reg S/I(X) = \max \set{\alpha_1(q-2),\left \lceil (n-1)(q-2)/2 \right \rceil}. 
\]
The case of the complete graph, \emph{i.e.}, when $\alpha_1=\cdots=\alpha_r=1$, is treated in \cite{GRS2}.
\smallskip

In the proof of this result we will need to show that $\reg S/I(X) \geq \left \lceil (n-1)(q-2)/2 \right \rceil$. This
inequality was shown to hold for any graph for which $\# X = (q-1)^{n-1}$ in \cite{GRS}. 
Indeed \cite[Corollary 3.13]{GRS} implies that if $X$ is the algebraic toric set associated 
to a $k$-uniform clutter on $n$ vertices then 
\[
\reg S/I(X)  \geq \left \lceil \frac{(\# X) (n-1)(q-2)}{k(q-1)^{n-1}}\right \rceil.
\]

In the proof of Theorem~\ref{thm: regularity} we argue on an Artinian reduction of $S/I(X)$. More precisely, if $f\in S$ 
is a regular element on $S/I(X)$ of degree $1$, the short exact sequence
\[
0\rt S/I(X) [-1] \stackrel{f}{\longrightarrow} S/I(X) \rt S/(I(X)+(f)) \rt 0 
\]
yields $\reg S/I(X) = \reg S/(I(X)+(f))-1$. Accordingly, showing that $\reg S/I(X) = d$ amounts to proving that 
every monomial of degree $d+1$ belongs to $I(X)+(f)$ and that there exists a monomial of degree $d$ that does not. 
For a detailed explanation of this fact, see \cite[Theorem 6.2]{evencycles}.
We begin with two lemmas. 

\begin{lemma}\label{lemma: moving the q-1 weight of a vertex}
Let $t^a\in S$ be a monomial and let $\H_{t^a}$ be the associated weighted subgraph of $\G$. Let 
$v_0\in V_\G$ and $\set{e_{i_k}=\set{v_0,v_{i_k}}:k=1,\dots,n'}$, $n' \geq 2$, be a 
subset of the edges of $\H_{t^a}$ incident to $v_0$, such
that $\sum_{k=1}^{n'} \wt_{\H_{t^a}}(e_{i_k})\geq \alpha(q-1)$, for some positive integer $\alpha$. 
Let $w_0\in V_\G\setminus\set{v_0}$ be such that $\set{w_0,v_{i_k}}\in E(\G)$, for all $k=1,\dots, n'$.
Then there exists a monomial $t^b\in S$ such that $\textstyle t^a - t^b \in I(X)$, 
$\wt_{\H_{t^b}}(v_0)=\wt_{\H_{t^a}}(v_0)-\alpha(q-1)$, 
$\wt_{\H_{t^b}}(w_0)\geq \sum_{k=1}^{n'} \wt_{\H_{t^b}}\set{w_0,v_{i_k}}\geq \alpha(q-1)$
and, for all $v\in V_\G\setminus\set{v_0,w_0}$, $\wt_{\H_{t^a}}(v) =\wt_{\H_{t^b}}(v)$.
\end{lemma}

\begin{proof}
We may write $\textstyle t^a = \bigl(\prod_{k=1}^{n'}t_{i_k}^{d_k}\bigr)t^{a'}$, 
where $\sum_{k=1}^{n'}d_k=\alpha(q-1)$, for some $a'\in \NN^s$. Denote $\set{w_0,v_{i_k}}$  
by $e_{j_1},\dots,e_{j_{n'}}$. Consider $t^b = \bigl(\prod_{k=1}^{n'}t_{j_k}^{d_k}\bigr)t^{a'}$. 
Then, by Lemma~\ref{lemma: generalised type III}, 
\[
\textstyle t^a - t^b =  
\bigl(\prod_{k=1}^{n'}t_{i_k}^{d_k} -\prod_{k=1}^{n'}t_{j_k}^{d_k}\bigr)t^{a'}\in I(X) . 
\]
Additionally, $\wt_{\H_{t^b}}(v_0)=\wt_{\H_{t^{a'}}}(v_0) = \wt_{\H_{t^a}}(v_0) - \alpha(q-1)$ and 
\[
\wt_{\H_{t^b}}(w_0) \geq
\textstyle \sum_{k=1}^{n'} \wt_{\H_{t^b}}\set{w_0,v_{i_k}}\geq \sum_{k=1}^{n'} d_k = \alpha(q-1)\geq q-1. 
\] 
Finally, if $v$ is not an endpoint of $e_{i_1},\dots,e_{i_{n'}}, e_{j_1},\dots,e_{j_{n'}}$ 
then $\wt_{\H_{t^a}}(v)=\wt_{\H_{t^{a'}}}(v)=\wt_{\H_{t^b}}(v)$,
and otherwise, 
$\wt_{\H_{t^a}}(v_{i_k})=d_k+\wt_{\H_{t^{a'}}}(v_{i_k})=\wt_{\H_{t^b}}(v_{i_k})$, for all $k=1,\dots,n'$.
\end{proof}

\begin{lemma}\label{lemma: swap endpoints of edges}
Let $t^a\in S$ be a monomial and let $\H_{t^a}$ be the associated weighted subgraph of $\G$. Given
$i,j\in \set{1,\dots,r}$ such that $i \neq j$, let $\textstyle \Delta^a_{j}$ 
be the total weight of edges between the vertices of $V_\G\setminus P_j$, let
$v_i\in P_i$ be a vertex, and let $\delta^a_{ij}\leq \Delta^a_{j}$ be the total weight of edges between 
$v_i$ and the vertices of $V_\G\setminus P_j$. 
Suppose that there exists an edge $\set{w_1,w_2}\in E(\H_{t^a})$ with 
$w_1,w_2\not \in P_j$ and $w_1,w_2 \neq v_i$.
Then, there exists a monomial $t^b\in S$ such that $t^a-t^b\in I(X)$, 
$\wt_{\H_{t^a}}(v)=\wt_{\H_{t^b}}(v)$, for all $v\in V_\G$, and such that the total
weight of edges between $v_i$ and the vertices of $V_\G\setminus P_j$, $\delta^b_{ij}$, is equal to 
$\min \bigl\{\wt_{\H_{t^a}}(v_i), \Delta^a_{j}\bigr\}$.
\end{lemma}

\begin{proof}
Notice that $\wt_{\H_{t^a}}(v_i) \geq \delta^a_{ij}$. We argue by induction on $\wt_{\H_{t^a}}(v_i)-\delta^a_{ij}$.
If $\wt_{\H_{t^a}}(v_i)=\delta^a_{ij}$, we choose $t^b=t^a$ and there is nothing to prove. 
Suppose that $\wt_{\H_{t^a}}(v_i) > \delta^a_{ij}$.
Then there exists an edge $\set{v_i,w_3}\in E(\H_{t^a})$ with $w_3\in P_j$.
Since $\set{w_1,w_2}$ is an edge, one of its vertices does not belong to $P_i$. Assume that $w_1\not \in P_i$.
Then $t_{\set{v_i,w_3}}t_{\set{w_1,w_2}}-t_{\set{v_i,w_1}}t_{\set{w_2,w_3}}$ is a generator of $I(X)$ of type II. 
Writing
$t^a = t_{\set{v_i,w_3}}t_{\set{w_1,w_2}}t^{a'}$, for suitable $a'\in \NN^s$, 
and $t^c=t_{\set{v_i,w_1}}t_{\set{w_2,w_3}}t^{a'}$,
it is clear that $t^a-t^c\in I(X)$; moreover the numbers $\Delta^c_{j}$ and $\wt_{\H_{t^c}}(v_i)$ 
are the same as they were 
for $\H_{t^a}$, however $\delta^c_{ij}=\delta^a_{ij}+1$. By induction, there exists a monomial $t^b \in S$ such that
$t^c-t^b\in I(X)$, hence $t^a-t^b\in I(X)$, 
and such that $\delta^b_{ij} = \min \bigl\{\wt_{\H_{t^c}}(v_i), \Delta^c_{j}\bigr\}= 
\min \bigl\{\wt_{\H_{t^a}}(v_i), \Delta^a_{j}\bigr\}$.
Also by induction, $\wt_{\H_{t^c}}(v)=\wt_{\H_{t^b}}(v)$, for all $v\in V_\G$.
On the level of the graph, the induction step  merely takes two edges and swaps a pair of their endpoints. 
This does not change the weighted degree; therefore $\wt_{\H_{t^c}}(v)=\wt_{\H_{t^a}}(v)$, for all $v\in V_\G$,
and hence, $\wt_{\H_{t^a}}(v)=\wt_{\H_{t^b}}(v)$, for all $v\in V_\G$.
\end{proof}

\begin{theorem}\label{thm: regularity}
Let $X$ be the algebraic toric set associated to an $r$-partite complete graph 
\mbox{$\G =\mathcal{K}_{\alpha_1,\dots,\alpha_r}$}, 
with $r\geq 3$ and $n=\alpha_1+\cdots+\alpha_r \geq 4$. Then 
\begin{equation*}
\reg S/I(X)=\max \set{\alpha_1(q-2),\alpha_2(q-2),\dots,\alpha_r(q-2),\left \lceil (n-1)(q-2)/2 \right \rceil}.
\end{equation*}
\end{theorem}

\begin{proof}
Fix vertices $v_i\in P_i\subset V_\G$. Without loss in generality, let $t_1$ be the variable $t_{\set{v_1,v_2}}$.
We assume that $\alpha_1\geq \alpha_2\geq \cdots \geq \alpha_r$. If $\alpha_1=1$ then $\G$ is the complete graph
of $n$ vertices and $\reg S/I(X) = \lceil (n-1)(q-2)/2 \rceil$, as was proved in \cite{GRS2}. From now on we assume that 
$\alpha_1 \geq 2$. By \cite[Corollary 3.13]{GRS}, 
$\reg S/I(X) \geq \left \lceil (n-1)(q-2)/2 \right \rceil$. Let us show that $\reg S/I(X) \geq  \alpha_1(q-2)$.
Consider the monomial of degree $\alpha_1(q-2)$ given by
\[
t^a = t_{\set{v_2,v_3}}^{q-2}\prod_{w\in P_1\setminus \set{v_1}} t_{\set{w,v_2}}^{q-2} \;,
\]
and consider $\H_{t^a}$, the weighted subgraph of $\G$ induced by $t^a$, which is depicted in 
Figure~\ref{fig: support of monomial}. 
\begin{small}
\begin{figure}[h]
\begin{picture}(200,125)(-50,55)
\put(-40,150){$\bullet$}
\put(-50,152){$v_1$}
\put(-40,130){$\bullet$}
\put(-38,133){\line(6,0){60}}
\put(-20,135){$\scriptstyle {q-2}$}
\put(-40,109){$\bullet$}
\put(-39,111){\line(3,1){60}}
\put(-22,122){$\scriptstyle {q-2}$}
\put(-39,87){$\vdots$}
\put(-40,70){$\bullet$}
\put(-38,71){\line(1,1){60}}
\put(-15,90){$\scriptstyle {q-2}$}
\put(-38,113){\oval(30,105)}
\put(-25,163){$P_1$}

\put(20,160){$\bullet$}
\put(21,148){$\vdots$}
\put(20,140){$\bullet$}
\put(20,130){$\bullet$}
\put(26,130){$v_2$}
\put(23,132){\line(0,-1){35}}
\put(25,113){$\scriptstyle {q-2}$}
\put(23,147){\oval(30,50)}
\put(35,170){$P_2$}

\put(10,95){$v_3$}
\put(20,95){$\bullet$}
\put(20,85){$\bullet$}
\put(21,73){$\vdots$}
\put(20,65){$\bullet$}
\put(23,82){\oval(30,50)}
\put(35,55){$P_3$}

\put(60,125){$v_4$}
\put(70,125){$\bullet$}
\put(70,115){$\bullet$}
\put(71,103){$\vdots$}
\put(70,95){$\bullet$}
\put(72,113){\oval(30,50)}
\put(85,135){$P_4$}

\put(96,113){$\dots$}

\put(120,125){$v_r$}
\put(130,125){$\bullet$}
\put(130,115){$\bullet$}
\put(131,103){$\vdots$}
\put(130,95){$\bullet$}
\put(132,113){\oval(30,50)}
\put(145,135){$P_r$}

\end{picture}
\caption{The weighted subgraph associated to $t^a$.}
\label{fig: support of monomial}
\end{figure}
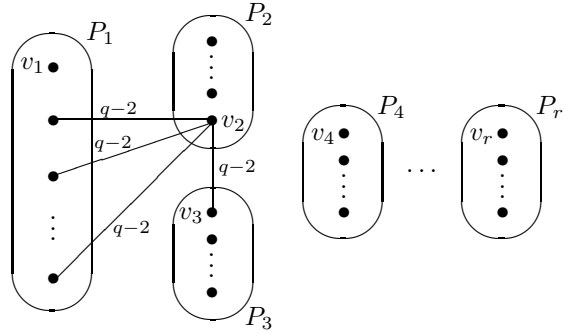
\end{small}

Let us show that  $t^a \not \in I(X)+(t_1)$. Let $\set{f_1,\dots,f_m}$ be a Gr\"obner basis
of $I(X)$. Since, by Theorem~\ref{thm: generators}, $I(X)$ is generated by binomials, we may assume
that $f_i$ is a binomial, for every $i=1,\dots,m$ (\emph{cf.}~\cite[Proposition 1.1 (a)]{ES}, a straighforward consequence of Buchberger's algoritm). 
Since a variable $t_i$ never vanishes on a point of
$X$, we may also assume that each $f_i=t^{a}-t^{b}$ with $\supp a\cap \supp b=\emptyset$. 
Choose a monomial order in which $t_1$ is the 
least variable. Then, since $f_1,\dots,f_m,t_1$ generate $I(X)+(t_1)$, and $t_1$ does not divide the leading
term of any $f_1,\dots,f_m$, we conclude, by Buchberger's algorithm 
(\emph{cf.}~\cite[Theorem~2.4.15]{monalg} and \cite{buch}), 
that $\set{t_1,f_1,\dots,f_m}$ is a Gr\"obner basis of $I(X)+(t_1)$. 
\smallskip

\noindent
Suppose that $t^a \in I(X)+(t_1)$. Then the division algorithm of $t^a$ by $t_1,f_1,\dots,f_m$ produces zero remainder.
Since $t_1$ does not divide $t^a$, there exists $i$ such that $\operatorname{Lt}(f_i)$ divides $t^a$. Write
$\operatorname{Lt}(f_i)= t^{b}$ and, without loss in generality, $f_i=t^{b}-t^{c}$, for some 
$b,c \in \NN^s$. Then there exists $a'\in \NN^s$ such that $t^a = t^b t^{a'}$. Let
\[
t^{a_1} = t^a - t^{a'}f_i = t^{a'} t^c 
\]
be the first partial reduction of $t^a$ modulo $t_1,f_1,\dots,f_m$.
Denote by $t^{a}=t^{a_0},t^{a_1},\dots,t^{a_p}$ 
the full sequence of partial reductions of $t^a$ modulo $t_1,f_1,\dots,f_m$. 
Since division of a monomial by one of $f_i$ always yields a nonzero monomial, we conclude that 
there must exist $p>0$ such that $t^{a_p}$ is divisible by $t_1$ (and, accordingly, $t^{a_{p+1}}=0$). 
Write $t^{a_p}=t_1t^w$, for some $w \in \NN^s$. 
From the fact that $t^{a_{i+1}}-t^{a_i}\in I(X)$, for all $i=0,\dots,p-1$, we 
deduce that 
\[
t^a-t_1t^w \in I(X). 
\]
Consider $\H_{t^a}$ and $\H_{t_1t^w}$ the weighted subgraphs of $\G$ induced by $t^a$ and $t_1t^w$, respectively. 
By the definition of $t^a$ (see also Figure~\ref{fig: support of monomial}), 
$\wt_{\H_{t^a}}(v_1)=0$ and, if $v\in P_1\setminus \set{v_1}$,
\mbox{$\wt_{\H_{t^a}}(v)= q-2$.} According\-ly, using  
Lemma~\ref{lemma: Vanishing lemma}, we deduce that $\wt_{\H_{t_1t^w}}(v_1)\geq q-1$ and, 
if $v\in P_1\setminus \set{v_1}$, that
\mbox{$\wt_{\H_{t_1t^w}}(v)\geq q-2$}. We get:
\[
\alpha_1(q-2)=\deg(t^a)=\deg(t_1t^w)\geq \textstyle\sum_{v\in P_1} \wt_{\H_{t_1t^w}}(v)\geq \alpha_1(q-2)+1, 
\]
which is a contradiction.
We just showed that $t^a\not \in I(X)+(t_1)$, hence $\reg S/I(X) \geq \alpha_1(q-2)$.

\smallskip

Let $d=\max \set{\alpha_1(q-2),\lceil (n-1)(q-2)/2 \rceil}$.
Let us now show that any monomial $t^a\in S$ of degree $d+1$ belongs to $I(X)+(t_1)$, or, equivalently, 
that there exists $t^b\in (t_1)$ such that $t^a\equiv t^b \mod I(X)$. As earlier, we will argue on $\H_{t^a}$,
the weighted subgraph of $\G$ induced by $t^a$.

\smallskip

Let $w_1, w_2 \in V_\G$ such that $\set{w_1,w_2}\in E(\H_{t^a})$. Write $t_{i_0}=t_{\set{w_1,w_2}}$
and $t^a=t_{i_0}^bt^{a'}$, where $b \geq 1$ and $a' \in \NN^s$. Suppose $b \geq q-1$.
If $t_{i_0}=t_1$, then $t^a \in (t_1)+ I(X)$. If $t_{i_0} \neq t_1$, then $g= t_{i_0}^{q-1}-t_1^{q-1}$ 
is a type I generator of $I(X)$. Writing $t^a=t_{i_0}^{q-1}t^{a''}$, for some $a'' \in \NN^s$, we have
$t^a=t_1^{q-1}t^{a''}+gt^{a''} \in (t_1)+I(X)$.
Therefore, we may assume that if $t_{i_0}^b$ divides $t^a$, then $b \leq q-2$. 
This implies we may also assume that if $v \in V_\G$ is such that $\wt_{\H_{t^a}}(v)\geq q-1$, then there exist at least
two edges in $\H_{t^a}$ that have $v$ as an endpoint.

\smallskip

We start by deriving a basic inequality. Suppose that $v_0\in P_i\setminus  \set{v_1,\dots,v_r}$ has 
$\wt_{\H_{t^a}}(v_0)\geq q-1$. 
Then using Lemma~\ref{lemma: moving the q-1 weight of a vertex} with $w_0=v_i$ and $\alpha$ 
a positive integer such that 
$(\alpha +1)(q-1)> \wt_{\H_{t^a}}(v_0)\geq \alpha (q-1)$, we see that there exists $t^b\in S$ such that 
$t^a\equiv t^b \mod I(X)$ and $\wt_{\H_{t^b}}(v_0)= \wt_{\H_{t^a}}(v_0)-\alpha(q-1)\leq q-2$.
Accordingly, we may assume that 
\begin{equation}\label{eq: first inequality}
\mbox{for all } v\not \in \set{v_1,\dots,v_r}, \; \wt_{\H_{t^a}}(v)\leq q-2.
\end{equation}

Hence,
$\textstyle \sum_{v\in V_\G} \wt_{\H_{t^a}}(v)\leq \sum_{j=1}^r \wt_{\H_{t^a}}(v_j) + (n-r)(q-2).$

\medskip

Since 
$\textstyle \sum_{v\in V_\G} \wt_{\H_{t^a}}(v) = 2(d+1)\geq 2\left \lceil \frac{n-1}{2}(q-2) \right \rceil +2 \geq (n-1)(q-2)+2$, 
we get
\begin{equation}\label{eq: basic inequality}
\textstyle \sum_{j=1}^r \wt_{\H_{t^a}}(v_j) \geq (r-1)(q-2)+2. 
\end{equation}

\emph{Case 1.} Suppose that $\wt_{\H_{t^a}}(v_1),\wt_{\H_{t^a}}(v_2)>0$. 

\emph{Subcase 1.1.} Suppose $v_1$ and $v_2$ have edges in $\H_{t^a}$ 
connecting them to $2$ vertices in distinct $P_i$, say $\set{v_1,w_1}, \set{v_2,w_2}$, 
with $w_1\in P_i$ and $w_2\in P_j$, and $i\not = j$.
If $w_1=v_2$ or $w_2=v_1$, then $t_1$ divides $t^a$ and we are done.
If $w_1 \neq v_2$ and $w_2 \neq v_1$, then since
$\set{w_1,w_2}$ belongs to  $E(\G)$, we have that $t_{\set{v_1,w_1}}t_{\set{v_2,w_2}}-t_1 t_{\set{w_1,w_2}}$ is a 
type II generator of $I(X)$. As, by assumption, $t_{\set{v_1,w_1}}t_{\set{v_2,w_2}}$ divides $t^a$, we may write 
$t^a = t_{\set{v_1,w_1}}t_{\set{v_2,w_2}}t^{a'}$, for a suitable $a' \in \NN^s$. Then,
$t^a \equiv t_1t_{\set{w_1,w_2}}t^{a'} \mod I(X)$,
as required. 

\emph{Subcase 1.2.} Suppose there exists $i\in \set{1,\dots,r}$ such that all edges in $\H_{t^a}$
from $v_1$ and $v_2$ have endpoints in $P_i$. Clearly $i\in \set{3,\dots,r}$. 
Assume, additionally, that there exists and edge 
$\set{w_1,w_2}\in E(\H_{t^a})$, such that $w_1,w_2\not \in P_i$. Then, necessarily $w_1,w_2 \not = v_1$ and
$w_1,w_2 \not = v_2$ and at least one of $\set{v_2,w_1}$ or $\set{v_2,w_2}$ belongs to $E(\G)$. 
Assume this is the case with $\set{v_2,w_1}$. Let $z\in P_i$ be 
an endpoint of an edge in $\H_{t^a}$ incident to $v_2$. 
Then $t_{\set{v_2,z}}t_{\set{w_1,w_2}}-t_{\set{v_2,w_1}}t_{\set{z,w_2}}$ is a type II generator of $I(X)$ and writing
$t^a=t_{\set{v_2,z}}t_{\set{w_1,w_2}}t^{a'}$, for a suitable $a' \in \NN^s$, we get
$t^a \equiv t_{\set{v_2,w_1}}t_{\set{z,w_2}}t^{a'} \mod I(X)$. Denote $t^e = t_{\set{v_2,w_1}}t_{\set{z,w_2}}t^{a'}$. 
Then $t^e$ satisfies the assumptions of Subcase 1.1 as $w_1\not \in P_i$, $\set{v_2,w_1}\in E(\H_{t^e})$, and
all the edges in $E(\H_{t^e})$ starting  at $v_1$ have the other endpoint in $P_i$.

\emph{Subcase 1.3.} Suppose there exists $i\in \set{3,\dots,r}$ 
such that all edges in $\H_{t^a}$ from $v_1$ and $v_2$ have endpoints in $P_i$.
Suppose in addition that all other edges in $\H_{t^a}$ are also incident to vertices of $P_i$. 
Using the assumption on $\deg(t^a)$ and (\ref{eq: first inequality}), we get 
\begin{equation}\label{eq: two inequalities}
\renewcommand{\arraystretch}{1.3}
\left \{
\begin{array}{l}
\sum_{v\in P_i} \wt_{\H_{t^a}}(v) = d+1 \geq \alpha_1(q-2) +1 \geq 2(q-2)+1 = (q-2)+(q-1)\\
\wt_{\H_{t^a}}(v_i)= d+1 -\sum_{v\in P_i\setminus \set{v_i}} \wt_{\H_{t^a}}(v)\geq \alpha_i(q-2)+1 -(\alpha_i -1)(q-2) = q-1. 
\end{array}
\right.
\end{equation}

The first inequality implies that either, $\Delta^a_{1}$, the total weight of the edges in $\H_{t^a}$ 
between the vertices of $V_\G\setminus P_1$ is $\geq q-1$, or $\Delta^a_{2}$, the total weight of 
edges in $\H_{t^a}$ between the vertices of $V_\G\setminus P_2$ is $\geq q-1$. 
Assume, without loss of generality that the latter is true. 

Let $\set{w_1,w_2}\in E(\H_{t^a})$ be an edge with $w_1,w_2\not \in P_2$ and assume $w_1,w_2 \neq v_i$.
Then, by Lemma~\ref{lemma: swap endpoints of edges} (with $j=2$), 
there exists $t^b \in S$ such that $t^a\equiv t^b \mod I(X)$,
$\wt_{\H_{t^b}}(v_1)=\wt_{\H_{t^a}}(v_1)>0$, and such that
the total weight of the edges of $\H_{t^b}$ between $v_i$ and $V_\G\setminus P_2$, 
$\delta^b_{i2}$, is equal to $\min \bigl\{\wt_{\H_{t^a}}(v_i), \Delta^a_{2}\bigr\}$.
From the second inequality of (\ref{eq: two inequalities}) and $\Delta^a_{2} \geq q-1$, we get 
$\delta^b_{i2}\geq q-1$. 
Observe also that the use of Lemma~\ref{lemma: swap endpoints of edges} guarantees that all
edges in $\H_{t^b}$ are still incident to vertices of $P_i$.

As done in a previous argument, if $t^b=t_k^lt^{b'}$, where $l \geq q-1$ and $b' \in \NN^s$, then
$t^b \in (t_1)+ I(X)$, and so, we may assume that if $t_k^l$ divides $t^b$, then $l \leq q-2$.
Since $\delta^b_{i2}\geq q-1$, there exist at least two edges in $\H_{t^b}$ that have $v_i$ as an
endpoint and such that the other endpoints are in $V_\G\setminus (P_2 \cup P_i)$.

We now use Lemma~\ref{lemma: moving the q-1 weight of a vertex} with $v_i$ and $v_2$ instead of $v_0$ and $w_0$. 
Then there exists a monomial $t^c \in S$ such that $t^b\equiv t^c \mod I(X)$,
$\wt_{\H_{t^c}}(v_1)=\wt_{\H_{t^b}}(v_1)>0$, and
$\wt_{\H_{t^c}}(v_2)\geq q-1$.

Choose $z \in V_\G\setminus (P_2 \cup P_i)$ as one of the endpoints mentioned above ($t_{\set{v_i,z}}$ dividing $t^b$),
and such that $t_{\set{v_i,z}}$ is used in the ``transfer of weight'' from $v_i$ to $v_2$ of 
Lemma~\ref{lemma: moving the q-1 weight of a vertex}. 
As a consequence, $t_{\set{z,v_2}}$ divides $t^c$. 



If $z=v_1$, then $t^c \in (t_1)$, and we are done.

Consider the case when $z \neq v_1$. 
Since $\wt_{\H_{t^c}}(v_1)>0$, there exists $u \in V_\G$ such that $t_{\set{v_1,u}}$ divides $t^c$.
If $u=v_2$, we are again done. If $u \neq v_2$, and since
all edges in $\H_{t^b}$ are incident to vertices of $P_i$, we must have
$u \in P_i$.

Now, $\H_{t^c}$ satisfies the assumptions of Subcase 1.1. 

We still need to consider the case when all edges $\set{w_1,w_2}\in E(\H_{t^a})$ with $w_1,w_2\not \in P_2$
are such that $w_1 = v_i$ or $w_2 = v_i$.
In this situation, $\delta^a_{i2} = \Delta^a_{2} \geq q-1$, and we repeat the above argument, using 
Lemma~\ref{lemma: moving the q-1 weight of a vertex} for $t^a$ instead of $t^b$.  
\smallskip

\emph{Case 2.} Suppose that $\wt_{\H_{t^a}}(v_1)\wt_{\H_{t^a}}(v_2)=0$. 

Assume, without loss of generality, that $\wt_{\H_{t^a}}(v_1)=0$
(if $\wt_{\H_{t^a}}(v_2)=0$, we argue in the same way, exchanging $v_1$ and $v_2$). 
Then from (\ref{eq: basic inequality}) we get 
\begin{equation}\label{eq: degenerate basic inequality}
\textstyle \sum_{i=2}^r  \wt_{\H_{t^a}} (v_i) \geq (r-1)(q-2) +2 \;.
\end{equation}
Denote by $\Delta^a_{1}$ the total weight of the edges in $\H_{t^a}$ between the vertices of $V_\G\setminus P_1$.
By (\ref{eq: first inequality}),
\[
\textstyle \alpha_1(q-2)+1 \leq d+1 = \sum_{v\in P_1} \wt_{\H_{t^a}}(v) +\Delta^a_1 \leq (\alpha_1-1)(q-2) +\Delta^a_1 \;, 
\]
and thus we deduce that $\Delta^a_1 \geq q-1$.

\emph{Subcase 2.1.} Assume $\wt_{\H_{t^a}}(v_2)\geq (q-1)+1$. 

Let $\set{w_1,w_2}\in E(\H_{t^a})$ be an edge with $w_1,w_2\not \in P_1$ and assume $w_1,w_2 \neq v_2$.
Then, Lemma~\ref{lemma: swap endpoints of edges} gives 
a monomial $t^b$ such that $t^a-t^b\in I(X)$,
$\wt_{\H_{t^b}}(v_1)=\wt_{\H_{t^a}}(v_1)=0$, $\wt_{\H_{t^b}}(v_2)=\wt_{\H_{t^a}}(v_2) \geq (q-1)+1$,
and such that the total weight of the edges from $v_2$ to the 
vertices of $V_\G\setminus P_1$, $\delta^b_{21}=\min \bigl\{\wt_{\H_{t^a}}(v_2), \Delta^a_{1}\bigr\}$, 
is $\geq q-1$. By Lemma~\ref{lemma: moving the q-1 weight of a vertex} 
(with $\alpha=1$, and $v_2$ and $v_1$ instead of $v_0$ and $w_0$), we get a monomial $t^c$ 
such that $t^b-t^c\in I(X)$,
$\wt_{\H_{t^c}}(v_2)=\wt_{\H_{t^b}}(v_2)-(q-1) \geq 1$ and $\wt_{\H_{t^c}}(v_1) \geq (q-1)$.
Therefore, $\H_{t^c}$ satisfies the assumptions of Case 1. 

If all edges $\set{w_1,w_2}\in E(\H_{t^a})$ with $w_1,w_2\not \in P_1$
are such that $w_1 = v_2$ or $w_2 = v_2$, then
$\delta^a_{21} = \Delta^a_{1} \geq q-1$, and repeating the argument (using 
Lemma~\ref{lemma: moving the q-1 weight of a vertex} for $t^a$ instead of $t^b$), we fall again in Case 1. 

\emph{Subcase 2.2.} Suppose that $1 \leq \wt_{\H_{t^a}}(v_2)\leq q-1$. Then (\ref{eq: degenerate basic inequality})
implies that there exists $i\in \set{3,\dots,r}$
such that $\wt_{\H_{t^a}}(v_i)\geq q-1$. We argue as in the previous subcase. 

Let $\set{w_1,w_2}\in E(\H_{t^a})$ be an edge with $w_1,w_2\not \in P_1$ and assume $w_1,w_2 \neq v_i$.
Then, Lemma~\ref{lemma: swap endpoints of edges} gives 
a monomial $t^b$ such that $t^a-t^b\in I(X)$,
$\wt_{\H_{t^b}}(v_1)=\wt_{\H_{t^a}}(v_1)=0$, $\wt_{\H_{t^b}}(v_2)=\wt_{\H_{t^a}}(v_2) \geq 1$,
and such that 
$\delta^b_{i1}=\min \bigl\{\wt_{\H_{t^a}}(v_i), \Delta^a_{1}\bigr\} \geq q-1$. 
By Lemma~\ref{lemma: moving the q-1 weight of a vertex} 
(with $\alpha=1$, and $v_i$ and $v_1$ instead of $v_0$ and $w_0$), we get a monomial $t^c$ 
such that $t^b-t^c\in I(X)$,
$\wt_{\H_{t^c}}(v_i)=\wt_{\H_{t^b}}(v_i)-(q-1)$, $\wt_{\H_{t^c}}(v_1) \geq (q-1)$,
and $\wt_{\H_{t^c}}(v_2)=\wt_{\H_{t^b}}(v_2) \geq 1$. 
Once again $\H_{t^c}$ satisfies the assumptions of Case 1. 

If all edges $\set{w_1,w_2}\in E(\H_{t^a})$ with $w_1,w_2\not \in P_1$
are such that $w_1 = v_i$ or $w_2 = v_i$, then
$\delta^a_{i1} = \Delta^a_{1} \geq q-1$, and repeating the argument (using 
Lemma~\ref{lemma: moving the q-1 weight of a vertex} for $t^a$ instead of $t^b$), we fall again in Case 1. 

\emph{Subcase 2.3.} Suppose that $\wt_{\H_{t^a}}(v_2)=0$. Then, as in Subcase 2.2, 
there exists $i\in \set{3,\dots,r}$ such that $\wt_{\H_{t^a}}(v_i)\geq q-1$. 
Since $\wt_{\H_{t^a}}(v_2)=0$, we can repeat for $\Delta^a_{2}$ what we did for $\Delta^a_{1}$:
\[
\textstyle \alpha_2(q-2)+1 \leq d+1 = \sum_{v\in P_2} \wt_{\H_{t^a}}(v) +\Delta^a_2 \leq (\alpha_2-1)(q-2) +\Delta^a_2 \;, 
\]
and conclude that $\Delta^a_2 \geq q-1$.
Using Lemmas~\ref{lemma: swap endpoints of edges}
and \ref{lemma: moving the q-1 weight of a vertex}, this time moving the weight of $v_i$ 
towards $v_2$, we can reduce to either
Subcase 2.1 or Subcase 2.2.

\medskip

This completes the case analysis. We conclude that every monomial $t^a$ of degree 
$d+1$, where \mbox{$d=\max\set{\alpha_1(q-2),\dots,\alpha_r(q-2),\lceil (n-2)(q-2)/2 \rceil}$}, 
belongs to $I(X)+(t_1)$; 
in other words, that $\reg S/I(X)\leq\max\set{\alpha_1(q-2),\dots,\alpha_r(q-2),\lceil (n-2)(q-2)/2 \rceil}$; 
which completes the proof of the theorem.
\end{proof}

\bibliographystyle{plain}

\end{document}